\documentclass[reqno, 11pt]{amsart}

\usepackage[text={155mm,235mm},centering]{geometry}
\input diagxy
\input xy

\usepackage[active]{srcltx} 

\newtheorem{thm}{Theorem}[section]
\newtheorem{cor}[thm]{Corollary}
\newtheorem{lem}[thm]{Lemma}
\newtheorem{prop}[thm]{Proposition}

\theoremstyle{definition}

\theoremstyle{property}

\theoremstyle{remark}
\newtheorem{rem}[thm]{Remark}

\numberwithin{equation}{section}
\usepackage[mathscr]{eucal}
\usepackage[all]{xy}
\usepackage{mathrsfs}
\usepackage{xypic}
\usepackage{amsfonts}
\usepackage{amsmath}
\usepackage{amsthm,bm}
\usepackage{amssymb,tikz-cd}
\usepackage{latexsym}
\usepackage{tabularx}
\usepackage{graphicx}
\usepackage{pict2e}
\usepackage[pagebackref]{hyperref}
\usepackage{tikz}
\usepackage{textcomp}
\usepackage[scr=rsfs]{mathalfa}
\definecolor{ceruleanblue}{rgb}{0.16, 0.32, 0.75}
\hypersetup{colorlinks=true,allcolors=ceruleanblue}
\allowdisplaybreaks

\begin{document}

\title[Hypercohomologies of truncated twisted holomorphic de Rham complexes]{Hypercohomologies of truncated twisted holomorphic de Rham complexes}

\author{Lingxu Meng}
\address{Department of Mathematics, North University of China, Taiyuan, Shanxi 030051,  P. R. China}
\email{menglingxu@nuc.edu.cn}%



\begin{abstract}
  We investigate the hypercohomologies of truncated twisted holomorphic de Rham complexes on (\emph{not necessarily compact}) complex manifolds. In particular, we generalize Leray-Hirsch, K\"{u}nneth and Poincar\'{e}-Serre duality theorems on them. At last,  a blowup formula is given, which  affirmatively  answers a question posed by Chen, Y. and Yang, S. in \cite{CY}. 
\end{abstract}

\maketitle

\section{Introduction}
All complex manifolds mentioned  in this paper are connected and not necessarily compact unless otherwise specified. Let $X$ be a complex manifold.
The de Rham theorem says that the  cohomology $H^k(X,\mathbb{C})$  with complex coefficients  can be computed by the hypercohomology $\mathbb{H}^k(X,\Omega_X^{\bullet})$ of the holomorphic de Rham complex $\Omega_X^{\bullet}$.
Moreover, if  $X$ is a compact K\"{a}hler manifold, the Hodge filtration $F^pH_{dR}^k(X,\mathbb{C})=\bigoplus\limits_{r\geq p}H_{\bar{\partial}}^{r,k-r}(X,\mathbb{C})$ on the de Rham cohomology $H_{dR}^k(X,\mathbb{C})$ is just the hypercohomology $\mathbb{H}^k(X,\Omega_X^{\geq p})$ of the truncated holomorphic de Rham complex $\Omega_X^{\geq p}$, which plays a significant role in the Hodge theory.
The Deligne cohomology is an important object in the research of algebraic cycles, since it connects with the intermediate Jacobian and the integeral Hodge class group.
It  
has a strong relation with the singular cohomology $H^*(X,\mathbb{Z})$  and the hypercohomology $\mathbb{H}^*(X,\Omega_X^{\leq p-1})$ of the truncated holomorphic de Rham complex $\Omega_X^{\leq p-1}$ via the long exact sequence
\begin{displaymath}
\tiny{\xymatrix{
\cdots\ar[r] & H^{k-1}(X,\mathbb{Z})\ar[r]^{} &\mathbb{H}^{k-1}(X,\Omega_X^{\leq p-1})\ar[r]^{} &H_D^{k}(X,\mathbb{Z}(p))\ar[r]^{ }&H^{k}(X,\mathbb{Z})\ar[r]^{}&\mathbb{H}^{k}(X,\Omega_X^{\leq p-1})\ar[r]&\cdots.
}}
\end{displaymath}
So it seems natural to investigate the hypercohomology of truncated holomorphic de Rham complexes.

We consider the more general cases - the hypercohomology $\mathbb{H}^*(X,\Omega_X^{[s,t]}(\mathcal{L}))$ of truncated twisted holomorphic de Rham complexes $\Omega_X^{[s,t]}(\mathcal{L})$ (see Section 3.2 or \cite[Section 4.2]{CY}  for the definition) and obtain the generalizations of several classical theorems for the de Rham cohomology with values in a local system as follows.
\begin{thm}[Leray-Hirsch theorem]\label{L-H}
Let $\pi:E\rightarrow X$ be a holomorphic fiber bundle with compact fibers over a complex manifold $X$ and  $\mathcal{L}$ a local system of $\underline{\mathbb{C}}_X$-modules of finite rank on $X$. Assume that there exist  $d$-closed  forms $t_1$, $\ldots$, $t_r$ of pure degrees on $E$ such that the restrictions  of their Dolbeault  classes $[t_1]_{\bar{\partial}}$, $\ldots$, $[t_r]_{\bar{\partial}}$ to $E_x$  is a basis of $H^{\bullet,\bullet}_{\bar{\partial}}(E_x)=\bigoplus\limits_{p,q\geq 0}H^{p,q}_{\bar{\partial}}(E_x)$ for every $x\in X$.
Then there exists an isomorphism
\begin{displaymath}
\bigoplus\limits_{i=1}^r\mathbb{H}^{k-u_i-v_i}(X,\Omega_X^{[s-u_i,t-u_i]}(\mathcal{L})) \tilde{\rightarrow} \mathbb{H}^{k}(E,\Omega_E^{[s,t]}(\pi^{-1}\mathcal{L}))
\end{displaymath}
for any $k$, $s$, $t$, where $(u_i,v_i)$ is the degree of $t_i$ for $1\leq i\leq r$.
\end{thm}

\begin{thm}[K\"{u}nneth theorem]\label{K}
Let $X$, $Y$ be complex manifolds and $\mathcal{L}$,  $\mathcal{H}$  local systems of $\underline{\mathbb{C}}_X$-, $\underline{\mathbb{C}}_Y$-modules of finite ranks on $X$, $Y$ respectively.      If $X$ or $Y$ is compact, then there is an isomorphism
\begin{displaymath}
\bigoplus\limits_{\substack{a+b=c\\u+w=s\\v+w=t}}\mathbb{H}^{a}(X,\Omega_X^{[u,v]}(\mathcal{L}))\otimes_{\mathbb{C}} \mathbb{H}^{b}(Y,\Omega_Y^{[w,w]}(\mathcal{H}))\cong \mathbb{H}^{c}(X\times Y,\Omega_{X\times Y}^{[s,t]}(\mathcal{L}\boxtimes\mathcal{H}))
\end{displaymath}
for any $c$, $s$, $t$.
\end{thm}

\begin{thm}[Poincar\'{e}-Serre duality theorem]\label{P-S}
Let $X$ be an $n$-dimensional compact complex manifold and $\mathcal{L}$ a local system of $\underline{\mathbb{C}}_X$-modules of finite rank on $X$.   Then there exists an isomorphism
\begin{displaymath}
\mathbb{H}^{k}(X,\Omega_X^{[s,t]}(\mathcal{L})) \cong (\mathbb{H}^{2n-k}(X,\Omega_X^{[n-t,n-s]}(\mathcal{L}^{\vee})))^*
\end{displaymath}
for any $k$, $s$, $t$, where $*$ denotes the dual of a $\mathbb{C}$-vector space.
\end{thm}

Recently, the blowup formulas for the de Rham cohomology with values in a local system were studied with different approaches  \cite{CY,M1,M2,M4,YZ,Z}.
More generally, Chen, Y. and Yang, S. posed a question on the existence of  a blowup formula for the hypercohomology of a truncated  twisted holomorphic de Rham complex (\cite[Question 10]{CY}). Now, we prove that it truly exists.
\begin{thm}\label{blowup}
Let $\pi:\widetilde{X}\rightarrow X$ be the blowup of a  complex manifold $X$ along a  complex submanifold $Y$ and $\mathcal{L}$ a local system of $\underline{\mathbb{C}}_X$-modules of finite rank on $X$. Assume that $i_Y:Y\rightarrow X$ is the inclusion and $r=\emph{codim}_{\mathbb{C}}Y\geq 2$. Then there exists an isomorphism
\begin{equation}\label{blowup0}
\mathbb{H}^k(X,\Omega_X^{[s,t]}(\mathcal{L}))\oplus \bigoplus_{i=1}^{r-1}\mathbb{H}^{k-2i}(Y,\Omega_Y^{[s-i,t-i]}(i_Y^{-1}\mathcal{L}))\cong \mathbb{H}^{k}(\widetilde{X},\Omega_{\widetilde{X}}^{[s,t]}(\pi^{-1}\mathcal{L})),
\end{equation}
for any $k$, $s$, $t$.
\end{thm}

In Section 2, we recall some basic notions on complexes and give some properties of them, which may be well known for experts. In Section 3, the double complexes $\mathcal{S}^{\bullet,\bullet}_{X}(\mathcal{L},s,t)$ and $\mathcal{T}^{\bullet,\bullet}_{X}(\mathcal{L},s,t)$ are defined and studied, which play important roles in the study of the hypercohomology $\mathbb{H}^{k}(X,\Omega_X^{[s,t]}(\mathcal{L}))$. In Section 4, Theorems \ref{L-H}-\ref{blowup}  are proved.

\subsection*{Acknowledgements}
 The author would like to thank Drs. Chen, Youming and Yang, Song  for many useful discussions.
 The author sincerely thanks the referee for many valuable suggestions and careful reading of the manuscript.
 The author is supported by the Natural Science Foundation of Shanxi Province of China (Grant No. 201901D111141).


\section{Complexes}
Let $\mathcal{C}$ be an Abelian category. All complexes and morphisms in this section are in $\mathcal{C}$.

\subsection{Complexes}
A \emph{complex} $(K^\bullet,d)$ consists of objects $K^{k}$  and morphisms $d^{k}:K^{k}\rightarrow K^{k+1}$ for all $k\in\mathbb{Z}$ satisfying that $d^{k+1}\circ d^{k}=0$. Sometimes, we briefly denote $(K^\bullet,d)$  by $K^\bullet$  and denote  its $k$-th cohomology by $H^k(K^\bullet)$.
A \emph{morphism} $f:K^{\bullet}\rightarrow L^{\bullet}$ of  complexes means a family of morphisms $f^{k}:K^{k}\rightarrow L^{k}$ for all $k\in \mathbb{Z}$ satisfying that $f^{k+1}\circ d^{k}_{K}=d^{k}_{L}\circ f^{k}$  for all $k\in \mathbb{Z}$.
Moreover, if there exists a morphism $g:L^{\bullet}\rightarrow K^{\bullet}$ satisfying that $f\circ g=\textrm{id}$ and $g\circ f=\textrm{id}$, then  $f$ is said to be an \emph{isomorphism}.

A \emph{double complex} $(K^{\bullet,\bullet},d_1,d_2)$ consists of objects $K^{p,q}$ and morphisms $d^{p,q}_1:K^{p,q}\rightarrow K^{p+1,q}$, $d^{p,q}_2:K^{p,q}\rightarrow K^{p,q+1}$ for all $(p,q)\in\mathbb{Z}^2$ satisfying that $d_1^{p+1,q}\circ d_1^{p,q}=0$, $d_2^{p,q+1}\circ d_2^{p,q}=0$ and $d_1^{p,q+1}\circ d_2^{p,q}+d_2^{p+1,q}\circ d_1^{p,q}=0$.
Sometimes, it will be shortly written as $K^{\bullet,\bullet}$.
Fixed  $p\in \mathbb{Z}$, $(K^{p,\bullet}, d_2^{p,\bullet})$  is a complex.
A \emph{morphism} $f:K^{\bullet,\bullet}\rightarrow L^{\bullet,\bullet}$ of double complexes means a family of morphisms $f^{p,q}:K^{p,q}\rightarrow L^{p,q}$ for all $p$, $q$ $\in \mathbb{Z}$ satisfying that $f^{p+1,q}\circ d^{p,q}_{K1}=d^{p,q}_{L1}\circ f^{p,q}$ and $f^{p,q+1}\circ d^{p,q}_{K2}=d^{p,q}_{L2}\circ f^{p,q}$ for all $p$, $q$ $\in \mathbb{Z}$. If there exists a morphism $g:L^{\bullet,\bullet}\rightarrow K^{\bullet,\bullet}$ satisfying that $f\circ g=\textrm{id}$ and $g\circ f=\textrm{id}$, then  $f$ is said to be an \emph{isomorphism}.

Let $K^{\bullet,\bullet}$ be a double complex.
The \emph{complex} $sK^{\bullet,\bullet}$ \emph{associated to} $K^{\bullet,\bullet}$ is defined as $(sK^{\bullet,\bullet})^k=\bigoplus\limits_{p+q=k}K^{p,q}$ and $d^k=\sum\limits_{p+q=k}d_1^{p,q}+\sum\limits_{p+q=k}d_2^{p,q}$.   Let $(_KE^{\bullet,\bullet}_r, F^\bullet H^\bullet)$ be the spectral sequence associated to a double complex $K^{\bullet,\bullet}$. Then the first page $_KE^{p,q}_1=H^q(K^{p,\bullet})$ and $H^k=H^k(sK^{\bullet,\bullet})$.
A morphism $f:K^{\bullet,\bullet}\rightarrow L^{\bullet,\bullet}$ of double complexes naturally induces a  morphism $sf:sK^{\bullet,\bullet}\rightarrow sL^{\bullet,\bullet}$ of  complexes, where $(sf)^k=\sum\limits_{p+q=k}f^{p,q}$ for any $k\in\mathbb{Z}$.

\subsection{Shifted complexes}
For an integer $m$, denote by $(K^{\bullet}[m],d[m])$ the \emph{$m$-shifted complex} of $(K^\bullet,d)$, namely,
$K^{\bullet}[m]^{k}=K^{k+m}$ and $d[m]^{k}=d^{k+m}$ for all $k\in\mathbb{Z}$. For a pair of integers $(m,n)$, denote by  $(K^{\bullet,\bullet}[m,n],d_1[m,n],d_2[m,n])$ the \emph{$(m,n)$-shifted double complex} of $(K^{\bullet,\bullet},d_1,d_2)$, namely, $K^{\bullet,\bullet}[m,n]^{p,q}=K^{p+m,q+n}$ and $d_i[m,n]^{p,q}=d_i^{p+m,q+n}$ for any $(p,q)\in \mathbb{Z}^2$ and $i=1$, $2$.
Clearly, $s(K^{\bullet,\bullet}[m,n])=(sK^{\bullet,\bullet})[m+n]$.

\subsection{Dual complexes}
In this subsection, assume that $\mathcal{C}$ is the category of $\mathbb{C}$-vector spaces.

For a $\mathbb{C}$-vector space $V$, let $V^*=Hom_{\mathbb{C}}(V,\mathbb{C})$ be its dual space. The \emph{dual  complex} $((K^{\bullet})^*,d^*)$ of $(K^{\bullet},d)$   is defined as $(K^{\bullet})^{*k}=(K^{-k})^*$ and $d^{*k}f^{k}=(-1)^{k+1}f^{k}\circ d^{-k-1}$ for any $k\in \mathbb{Z}$.

Let $f^k\in (K^\bullet)^{*k}$ satisfy $d^{*k}f^k=0$, i.e., $f^k\circ d^{-k-1}=0$. Then $f^k$ naturally induces a linear map $\tilde{f}^k:H^{-k}(K^\bullet)\rightarrow \mathbb{C}$. The morphism $[f^k]\mapsto \tilde{f}^k$ gives the following isomorphism, where $[f^k]$ denotes the class in $H^{k}((K^\bullet)^*)$.
\begin{lem}[{\cite[IV, (16.5)]{Dem}}]\label{Duality}
 $H^{k}((K^\bullet)^*)\cong (H^{-k}(K^\bullet))^*$.
\end{lem}
The \emph{dual double complex} $((K^{\bullet,\bullet})^*,d^*_1,d^*_2)$ of $(K^{\bullet,\bullet},d_1,d_2)$  is defined as $(K^{\bullet,\bullet})^{*p,q}=(K^{-p,-q})^*$ and $d_1^{*p,q}f^{p,q}=(-1)^{p+q+1}f^{p,q}\circ d_1^{-p-1,-q}$, $d_2^{*p,q}f^{p,q}=(-1)^{p+q+1}f^{p,q}\circ d_2^{-p,-q-1}$ for any $(p,q)\in\mathbb{Z}^2$ and $f^{p,q}\in (K^{-p,-q})^*$.
\begin{lem}\label{simple-duality}
Suppose that the double complex $K^{\bullet,\bullet}$  is bounded, i.e., $K^{\bullet,\bullet}=0$ except finite $(p,q)\in\mathbb{Z}^2$. Then there exists an isomorphism  $(sK^{\bullet,\bullet})^*\cong s((K^{\bullet,\bullet})^*)$.
\end{lem}
\begin{proof}
For every $k$, define  a map $(sK^{\bullet,\bullet})^{*k}\rightarrow s((K^{\bullet,\bullet})^*)^k$ by $f\mapsto (f|_{K^{-p,-q}})_{p+q=k}$, where $f\in (sK^{\bullet,\bullet})^{*k}=Hom_{\mathbb{C}}(\bigoplus\limits_{p+q=k}K^{-p,-q},\mathbb{C})$. Since $K^{\bullet,\bullet}$  is bounded, this map gives an isomorphism $(sK^{\bullet,\bullet})^*\cong s((K^{\bullet,\bullet})^*)$.
\end{proof}

\subsection{Tensors of complexes}
In this subsection, assume that $\mathcal{C}$ is the category of $\mathbb{C}$-vector spaces.

For complexes $(K^{\bullet},d_K)$, $(L^{\bullet},d_L)$ and an integer $m$, the  double  complex $(K^\bullet\otimes_{\mathbb{C}}L^{\bullet})_m$ is defined as
$(K^\bullet\otimes_{\mathbb{C}}L^\bullet)_m^{p,q}=K^{p}\otimes_{\mathbb{C}} L^{q}$ and
$d_1^{p,q}=d_K^{p}\otimes 1_{L^q}, d^{p,q}_2=(-1)^{m+p}1_{K^p}\otimes d_L^{q}$ for any $(p,q)\in\mathbb{Z}^2$.
Clearly, $(K^\bullet\otimes_{\mathbb{C}}L^{\bullet})_m$ is equal to $(K^\bullet\otimes_{\mathbb{C}}L^{\bullet})_0$ for even $m$ and  to $(K^\bullet\otimes_{\mathbb{C}}L^{\bullet})_1$ for odd $m$.
\begin{lem}\label{isom}
There is an isomorphism $(K^\bullet\otimes_{\mathbb{C}}L^{\bullet})_0\cong(K^\bullet\otimes_{\mathbb{C}}L^{\bullet})_1$ of double complexes.
\end{lem}
\begin{proof}
Evidently, linear extending $\alpha^p\otimes\beta^q\mapsto(-1)^q\alpha^p\otimes\beta^q$ for $\alpha^p\in K^p$ and $\beta^p\in L^q$ gives an isomorphism $(K^\bullet\otimes_{\mathbb{C}}L^{\bullet})_0\rightarrow(K^\bullet\otimes_{\mathbb{C}}L^{\bullet})_1$ of double complexes.
\end{proof}

Combining \cite[IV, (15.6)]{Dem} and Lemma \ref{isom}, we conclude that
\begin{lem}\label{A-K}
For any integer $m$,
$\bigoplus\limits_{p+q=k}H^{p}(K^\bullet)\otimes_{\mathbb{C}} H^{q}(L^\bullet)\cong H^k(s((K^\bullet\otimes_{\mathbb{C}}L^\bullet)_m))$.
\end{lem}

For double complexes $(K^{\bullet,\bullet},d_{K1},d_{K2})$ and $(L^{\bullet,\bullet},d_{L1},d_{L2})$,  set
\begin{displaymath}
A^{p,q;r,s}=K^{p,r}\otimes_{\mathbb{C}} L^{q,s},
\end{displaymath}
\begin{displaymath}
d_1^{p,q;r,s}=d_{K1}^{p,r}\otimes 1_{L^{q,s}}:A^{p,q;r,s}\rightarrow A^{p+1,q;r,s},
\end{displaymath}
\begin{displaymath}
d_2^{p,q;r,s}=(-1)^{p+r}1_{K^{p,r}}\otimes d_{L1}^{q,s}:A^{p,q;r,s}\rightarrow A^{p,q+1;r,s},
\end{displaymath}
\begin{displaymath}
d_3^{p,q;r,s}=d_{K2}^{p,r}\otimes 1_{L^{q,s}}:A^{p,q;r,s}\rightarrow A^{p,q;r+1,s},
\end{displaymath}
\begin{displaymath}
d_4^{p,q;r,s}=(-1)^{p+r}1_{K^{p,r}}\otimes d_{L2}^{q,s}:A^{p,q;r,s}\rightarrow A^{p,q;r,s+1},
\end{displaymath}
for any $p$, $q$, $s$, $r$. We easily see that
\begin{equation}\label{partial1}
(A^{p,q;\bullet,\bullet},d_3^{p,q;\bullet,\bullet},d_4^{p,q;\bullet,\bullet})= (K^{p,\bullet}\otimes_{\mathbb{C}}L^{q,\bullet})_p
\end{equation}
as double complexes for any $p$, $q$.

Define $ss(K^{\bullet,\bullet}\otimes_{\mathbb{C}}L^{\bullet,\bullet})$   as
\begin{displaymath}
ss(K^{\bullet,\bullet}\otimes_{\mathbb{C}}L^{\bullet,\bullet})^{k,l}=\bigoplus\limits_{\substack{p+q=k\\r+s=l}}A^{p,q;r,s}
\end{displaymath}
and
\begin{displaymath}
D_{1}^{k,l}=\sum\limits_{\substack{p+q=k\\r+s=l}} d_1^{p,q;r,s}+ \sum\limits_{\substack{p+q=k\\r+s=l}} d_2^{p,q;r,s}    :ss(K^{\bullet,\bullet}\otimes_{\mathbb{C}}L^{\bullet,\bullet})^{k,l}\rightarrow ss(K^{\bullet,\bullet}\otimes_{\mathbb{C}}L^{\bullet,\bullet})^{k+1,l}
\end{displaymath}
\begin{displaymath}
D_2^{k,l}=\sum\limits_{\substack{p+q=k\\r+s=l}} d_3^{p,q;r,s}+ \sum\limits_{\substack{p+q=k\\r+s=l}} d_4^{p,q;r,s}    :ss(K^{\bullet,\bullet}\otimes_{\mathbb{C}}L^{\bullet,\bullet})^{k,l}\rightarrow ss(K^{\bullet,\bullet}\otimes_{\mathbb{C}}L^{\bullet,\bullet})^{k,l+1}
\end{displaymath}
for any $k$, $l$. it is easily checked that $ss(K^{\bullet,\bullet}\otimes_{\mathbb{C}}L^{\bullet,\bullet})$ is a double complex.

By (\ref{partial1}), the following lemma holds.
\begin{lem}\label{partial2}
For any $k$, there is an isomorphism
\begin{displaymath}
(ss(K^{\bullet,\bullet}\otimes_{\mathbb{C}}L^{\bullet,\bullet})^{k,\bullet},D_2^{k,\bullet})\cong \bigoplus\limits_{p+q=k}s((K^{p,\bullet}\otimes_{\mathbb{C}}L^{q,\bullet})_p)
\end{displaymath}
of complexes.
\end{lem}
By definitions, it follows that
\begin{lem}\label{simple-simple-simple}
$s(ss(K^{\bullet,\bullet}\otimes_{\mathbb{C}}L^{\bullet,\bullet}))\cong s((sK^{\bullet,\bullet}\otimes_{\mathbb{C}}sL^{\bullet,\bullet})_0)$.
\end{lem}
Now, we can compute the cohomologies of tensors of double complexes.
\begin{prop}\label{cohomology of complexes}
$(1)$ For any $a\in\mathbb{Z}$,
\begin{displaymath}
\bigoplus\limits_{k+l=a}H^{k}(sK^{\bullet,\bullet})\otimes_{\mathbb{C}} H^{l}(sL^{\bullet,\bullet})\cong H^{a}(s(ss(K^{\bullet,\bullet}\otimes_{\mathbb{C}}L^{\bullet,\bullet}))).
\end{displaymath}

$(2)$ For any $k$, $l\in\mathbb{Z}$,
\begin{displaymath}
\bigoplus\limits_{\substack{p+q=k\\r+s=l}}H^{r}(K^{p,\bullet})\otimes_{\mathbb{C}} H^{s}(L^{q,\bullet})\cong H^{l}(ss(K^{\bullet,\bullet}\otimes_{\mathbb{C}}L^{\bullet,\bullet})^{k,\bullet}).
\end{displaymath}
\end{prop}
\begin{proof}
By Lemma \ref{A-K} and \ref{simple-simple-simple}, we get $(1)$. By Lemma \ref{A-K} and \ref{partial2}, we get $(2)$.
\end{proof}

\section{Truncated twisted holomorphic de Rham complex}
Let $X$ be an $n$-dimensional complex manifold.
Denote by $\mathcal{A}_X^{p,q}$ (resp. $\mathcal{A}_X^{k}$, $\mathcal{D}_X^{\prime p,q}$, $\Omega_X^p$) the sheaf of germs of smooth $(p,q)$-forms (resp. complex-valued smooth $k$-forms, $(p,q)$-currents, holomorphic $p$-forms) on $X$.
Denote by $\underline{\mathbb{C}}_X$ the constant sheaf with stalk $\mathbb{C}$ over $X$.
Let $\mathcal{L}$ be a local system of $\underline{\mathbb{C}}_X$-modules of finite rank  on $X$, namely, a locally constant sheaf of finite dimensional $\mathbb{C}$-vector spaces on $X$.
Tensoring $\otimes_{\underline{\mathbb{C}}_X}$ between sheaves of $\underline{\mathbb{C}}_X$-modules will be simply written as $\otimes$.
For a sheaf $\mathcal{F}$ over $X$ and an open subset $U\subseteq X$, $\Gamma(U, \mathcal{F})$ refers to the group of sections of $\mathcal{F}$ on $U$. For a complex $\mathcal{F}^\bullet$ of sheaves on $X$, $\mathbb{H}^k(X,\mathcal{F}^\bullet)$ denotes its $k$-th hypercohomology.

\subsection{Twisted Dolbeault cohomology}
Since $\partial:\mathcal{A}_X^{p,q}\rightarrow \mathcal{A}_X^{p+1,q}$  and $\bar{\partial}:\mathcal{A}_X^{p,q}\rightarrow \mathcal{A}_X^{p,q+1}$ are both morphisms of sheaves of $\underline{\mathbb{C}}_X$-modules, they naturally induce $\mathcal{L}\otimes\mathcal{A}_X^{p,q}\rightarrow \mathcal{L}\otimes\mathcal{A}_X^{p+1,q}$ and $\mathcal{L}\otimes\mathcal{A}_X^{p,q}\rightarrow \mathcal{L}\otimes\mathcal{A}_X^{p,q+1}$, which are still denoted by $\partial$ and $\bar{\partial}$ respectively.
Then $\mathcal{L}\otimes\Omega_X^p$ has a fine resolution
\begin{displaymath}
\xymatrix{
0\ar[r] &\mathcal{L}\otimes\Omega_X^p\ar[r]^{i} &\mathcal{L}\otimes\mathcal{A}_{X}^{p,0}\ar[r]^{\quad\bar{\partial}}&\cdots\ar[r]^{\bar{\partial}\quad\quad}&\mathcal{L}\otimes\mathcal{A}_{X}^{p,n}\ar[r]&0.
}
\end{displaymath}
These arguments also hold for the sheaves $\mathcal{D}_X^{\prime \bullet,\bullet}$. 
So
\begin{displaymath}
H^q(X,\mathcal{L}\otimes\Omega_X^p)\cong H^q(\Gamma(X,\mathcal{L}\otimes\mathcal{A}_{X}^{p,\bullet}))\cong H^q(\Gamma(X,\mathcal{L}\otimes\mathcal{D}_{X}^{\prime p,\bullet})),
\end{displaymath}
which are uniformly  called the \emph{twisted Dolbeault cohomologies}.

\subsection{Double complexes $\mathcal{S}^{\bullet,\bullet}_{X}(\mathcal{L},s,t)$ and $\mathcal{T}^{\bullet,\bullet}_{X}(\mathcal{L},s,t)$}
In \cite[Section 4.2]{CY}, Chen, Y. and Yang, S. introduced the complexes $\Omega_X^{[s,t]}(\mathcal{L})$  for $0\leq s\leq t\leq n$  on a compact complex manifold $X$.  For convenience, we generalize this concept a little.

\emph{The compactness of $X$ is not assumed in this paper}. Set  $\mathcal{L}\otimes\Omega_X^k=0$ for $k<0$ or $>n$.
Given any integers $s$ and $t$, the \emph{truncated twisted holomorphic de Rham complex} $\Omega_X^{[s,t]}(\mathcal{L})$ is defined as the zero complex if $s> t$ and as the complex
\begin{equation}\label{truncated twisted}
\xymatrix{
0\ar[r] &\mathcal{L}\otimes\Omega_X^s\ar[r]^{\partial} &\mathcal{L}\otimes\Omega_X^{s+1}\ar[r]^{ \quad\partial}&\cdots\ar[r]^{\partial\quad}&\mathcal{L}\otimes\Omega_X^t\ar[r]&0
}
\end{equation}
if $s\leq t$, where $\mathcal{L}\otimes\Omega_X^k$ is placed in degree $k$ for $s\leq k\leq t$ and zeros are placed in other degrees.
Obviously, $\Omega_X^{[s,t]}(\mathcal{L})$ is equal to $\Omega_X^{[0,t]}(\mathcal{L})$ for $s<0$ and equal to $\Omega_X^{[s,n]}(\mathcal{L})$ for $t>n$. In particular, $\Omega_X^{[0,n]}(\mathcal{L})=\mathcal{L}\otimes\Omega_X^{\bullet}$ is the twisted holomorphic de Rham complex on $X$ and $\Omega_X^{[p,p]}(\mathcal{L})=(\mathcal{L}\otimes\Omega_X^{p\bullet})[-p]$, where $\Omega_X^{p\bullet}$ denote the complex with $\Omega_X^{p}$ in degree $0$ and zeros in other degrees.


Set
\begin{displaymath}
\mathcal{S}^{p,q}_{X}(\mathcal{L},s,t)=\left\{
 \begin{array}{ll}
\mathcal{L}\otimes\mathcal{A}^{p,q}_{X},&~s\leq p\leq t\\
 &\\
 0,&~\textrm{others}.
 \end{array}
 \right.
\end{displaymath}
Then $(\mathcal{S}^{\bullet,\bullet}_{X}(\mathcal{L},s,t),d_1,d_2)$ is a double complex of sheaves, where
\begin{displaymath}
d^{p,q}_1=\left\{
 \begin{array}{ll}
\partial,&~s\leq p< t\\
 &\\
 0,&~\textrm{others}
 \end{array}
 \right.
\textrm{and}\quad
d^{p,q}_2=\left\{
 \begin{array}{ll}
\bar{\partial},&~s\leq p\leq t\\
 &\\
 0,&~\textrm{others}.
 \end{array}
 \right.
\end{displaymath}
Denoted it by $\mathcal{S}^{\bullet,\bullet}_{X}(\mathcal{L},s,t)$ shortly.
Let $\mathcal{S}^{\bullet}_{X}(\mathcal{L},s,t)=s\mathcal{S}^{\bullet,\bullet}_{X}(\mathcal{L},s,t)$  be the complex associated to $\mathcal{S}^{\bullet,\bullet}_{X}(\mathcal{L},s,t)$.
For instance, $\mathcal{S}^{\bullet,\bullet}_{X}(\mathcal{L},0,n)=\mathcal{L}\otimes\mathcal{A}_X^{\bullet,\bullet}$, $\mathcal{S}^{\bullet}_{X}(\mathcal{L},0,n)=\mathcal{L}\otimes\mathcal{A}_{X}^{\bullet}$ and $\mathcal{S}^{\bullet}_{X}(\mathcal{L},p,p)=\mathcal{L}\otimes\mathcal{A}_X^{p,\bullet}[-p]$.
\begin{lem}\label{quasi-isom}
The inclusion gives a quasi-isomorphism $\Omega_X^{[s,t]}(\mathcal{L})\rightarrow \mathcal{S}^{\bullet}_{X}(\mathcal{L},s,t)$ of complexes of sheaves.
\end{lem}
\begin{proof}
For any $p\in\mathbb{Z}$, $\Omega_X^{[s,t]}(\mathcal{L})^p\rightarrow (\mathcal{S}^{p,\bullet}_{X}(\mathcal{L},s,t),d^{p,\bullet}_2)$ given by the inclusion is a resolution of $\Omega_X^{[s,t]}(\mathcal{L})^p$.
By \cite[Lemma 8.5]{V}, the lemma holds.
\end{proof}

Set $S^{p,q}(X,\mathcal{L},s,t)=\Gamma(X,\mathcal{S}^{p,q}_{X}(\mathcal{L},s,t))$ and $S^{p}(X,\mathcal{L},s,t)=\Gamma(X,\mathcal{S}^{p}_{X}(\mathcal{L},s,t))$. In particular, the double complex $S^{\bullet,\bullet}(X,\mathcal{L},0,n)$ is just $(\Gamma(X,\mathcal{L}\otimes \mathcal{A}_X^{\bullet,\bullet}),\partial,\bar{\partial})$ (see also \cite[Section 2.1]{CY} or \cite[Section 8.2.1]{V}).
For any $p\in\mathbb{Z}$, $\mathcal{S}^{p}_{X}(\mathcal{L},s,t)$ is $\Gamma$-acyclic, since it is a fine sheaf.
By \cite[Proposition 8.12]{V} and Lemma \ref{quasi-isom},
\begin{equation}\label{computation1}
\mathbb{H}^k(X,\Omega_X^{[s,t]}(\mathcal{L}))\cong H^k(S^{\bullet}(X,\mathcal{L},s,t))
\end{equation}
for any $k\in\mathbb{Z}$.
For example, $\mathbb{H}^k(X,\Omega_X^{[0,n]}(\mathcal{L}))\cong H_{dR}^k(X,\mathcal{L})$ and $\mathbb{H}^k(X,\Omega_X^{[p,p]}(\mathcal{L}))\cong H^{k-p}(X,\mathcal{L}\otimes\Omega^p_X)$.

Similarly, we can define $\mathcal{T}^{\bullet,\bullet}_{X}(\mathcal{L},s,t)$, $\mathcal{T}^{\bullet}_{X}(\mathcal{L},s,t)$, $T^{\bullet,\bullet}(X,\mathcal{L},s,t)$ and $T^{\bullet}(X,\mathcal{L},s,t)$, where
\begin{displaymath}
\mathcal{T}^{p,q}_{X}(\mathcal{L},s,t)=\left\{
 \begin{array}{ll}
\mathcal{L}\otimes\mathcal{D}^{\prime p,q}_{X},&~s\leq p\leq t\\
 &\\
 0,&~\textrm{others}.
 \end{array}
 \right.
\end{displaymath}
There is an isomorphism
\begin{equation}\label{computation2}
\mathbb{H}^k(X,\Omega_X^{[s,t]}(\mathcal{L}))\cong H^k(T^{\bullet}(X,\mathcal{L},s,t))
\end{equation}
for any $k\in\mathbb{Z}$.

\emph{Notice that}, all double complexes defined in this section are bounded.

\subsection{Exact sequences}
For integers $s\leq s'\leq t\leq t'$, there is a natural exact sequence
\begin{displaymath}
0\rightarrow S^{\bullet,\bullet}(X,\mathcal{L},t+1,t')\rightarrow S^{\bullet,\bullet}(X,\mathcal{L},s',t')\rightarrow S^{\bullet,\bullet}(X,\mathcal{L},s,t)\rightarrow S^{\bullet,\bullet}(X,\mathcal{L},s,s'-1)\rightarrow 0,
\end{displaymath}
where the morphism is the identity or zero at every degree.
In particular, we have a short exact sequence
\begin{equation}\label{exact}
0\rightarrow S^{\bullet,\bullet}(X,\mathcal{L},s,t)\rightarrow S^{\bullet,\bullet}(X,\mathcal{L},r,t)\rightarrow S^{\bullet,\bullet}(X,\mathcal{L},r,s-1)\rightarrow 0,
\end{equation}
for integers $r\leq s\leq t$. These exact sequences also hold for $\Omega_X^{[s,t]}$,
$S^{\bullet}(X,\mathcal{L},s,t)$,
$T^{\bullet}(X,\mathcal{L},s,t)$, etc.
By (\ref{computation1}) and (\ref{exact}), there is a long exact sequence
\begin{displaymath}
\tiny{\xymatrix{
\cdots\ar[r] &\mathbb{H}^{k-1}(X,\Omega_X^{[r,s-1]}(\mathcal{L}))\ar[r]^{} &\mathbb{H}^k(X,\Omega_X^{[s,t]}(\mathcal{L}))\ar[r]^{}&\mathbb{H}^k(X,\Omega_X^{[r,t]}(\mathcal{L}))\ar[r]^{}&\mathbb{H}^k(X,\Omega_X^{[r,s-1]}(\mathcal{L}))\ar[r]^{}&\cdots
}}.
\end{displaymath}

\subsection{Operations}
Suppose that $X$ is a  complex manifold and $\mathcal{L}$, $\mathcal{H}$ are local systems of $\underline{\mathbb{C}}_X$-modules of rank $l$, $h$ on $X$ respectively.

For $\alpha\in\Gamma(X,\mathcal{L}\otimes\mathcal{A}_X^{p})$ and $\beta\in\Gamma(X,\mathcal{H}\otimes\mathcal{A}_X^{q})$, we define the wedge product  $\alpha\wedge\beta\in\Gamma(X,\mathcal{L}\otimes\mathcal{H}\otimes\mathcal{A}_X^{p+q})$ as follows: Let $U$ be an open subset of $X$ such that $\mathcal{L}|_U$ and $\mathcal{H}|_U$ are trivial. Let $e_1$, $\ldots$, $e_l$ and $f_1$, $\ldots$, $f_h$ be bases of $\Gamma(U,\mathcal{L})$ and $\Gamma(U, \mathcal{H})$, respectively.
Suppose that $\alpha=\sum\limits_{i=0}^{l}e_i\otimes \alpha_i$ and $\beta=\sum\limits_{j=1}^{h}f_j\otimes \beta_i$ on $U$ respectively, where $\alpha_i\in\mathcal{A}^{p}(U)$ for $1\leq i\leq l$ and $\beta_i\in\mathcal{A}^{q}(U)$ for $1\leq j\leq h$ . Then $\alpha\wedge\beta$ on $U$ is defined as
\begin{displaymath}
\sum\limits_{\substack{1\leq i\leq l\\1\leq j\leq h}} e_i\otimes f_j\otimes(\alpha_i\wedge\beta_j).
\end{displaymath}
This construction is global. We have
\begin{displaymath}
d(\alpha\wedge\beta)=d\alpha\wedge\beta+(-1)^p\alpha\wedge d\beta,
\end{displaymath}
which also holds if using $\partial$ or $\bar{\partial}$ instead of $d$.

Suppose that $\alpha\in\Gamma(X,\mathcal{H}\otimes \mathcal{A}_X^{p,q})$ satisfies that $\partial\alpha=\bar{\partial}\alpha=0$, i.e., $d\alpha=0$. Then wedge products by $\alpha$ give a morphism 
\begin{displaymath}
\bullet\wedge\alpha:S^{\bullet,\bullet}(X,\mathcal{L},s,t)\rightarrow S^{\bullet,\bullet}(X,\mathcal{L}\otimes\mathcal{H},s+p,t+p)[p,q]
\end{displaymath}
of double complexes of sheaves. This  define a \emph{cup product}
\begin{displaymath}
\cup:\mathbb{H}^{k}(X,\Omega_X^{[s,t]}(\mathcal{L}))\times H^{p,q}_{BC}(X,\mathcal{H})\rightarrow \mathbb{H}^{k+p+q}(X,\Omega_X^{[s+p,t+p]}(\mathcal{L}\otimes \mathcal{H}))
\end{displaymath}
for any $k$, where
\begin{displaymath}
H^{p,q}_{BC}(X,\mathcal{H}):= \frac{Ker(d: \Gamma(X,\mathcal{H}\otimes\mathcal{A}_X^{p,q})\rightarrow \Gamma(X,\mathcal{H}\otimes\mathcal{A}_X^{p+q+1}))}{\partial\overline{\partial}\Gamma(X,\mathcal{H}\otimes\mathcal{A}_X^{p-1,q-1})}
\end{displaymath}
$H^{p,q}_{BC}(X,\mathcal{H})$ is \emph{the twisted Bott-Chern cohomology} of $X$.
We can obtain the same cup product by
\begin{displaymath}
\bullet\wedge\alpha:T^{\bullet,\bullet}(X,\mathcal{L},s,t)\rightarrow T^{\bullet,\bullet}(X,\mathcal{L}\otimes\mathcal{H},s+p,t+p)[p,q].
\end{displaymath}
Let $f:Y\rightarrow X$ be a holomorphic map between complex manifolds. The pullbacks give a morphism
\begin{displaymath}
f^*:S^{\bullet,\bullet}(X,\mathcal{L},s,t)\rightarrow S^{\bullet,\bullet}(Y,f^{-1}\mathcal{L},s,t)
\end{displaymath}
of double complexes, which induces a morphism
$f^*:\mathbb{H}^{k}(X,\Omega_X^{[s,t]}(\mathcal{L}))\rightarrow \mathbb{H}^{k}(Y,\Omega_Y^{[s,t]}(f^{-1}\mathcal{L}))$  for any $k$.
Moreover, if $f$ is proper and $r=\textrm{dim}_{\mathbb{C}}Y-\textrm{dim}_{\mathbb{C}}X$, then the pushforwards give a morphism
\begin{displaymath}
f_*:T^{\bullet,\bullet}(Y,f^{-1}\mathcal{L},s,t)\rightarrow T^{\bullet,\bullet}(X,\mathcal{L},s-r,t-r)[-r,-r]
\end{displaymath}
of double complexes, which induces a morphism
$f_*:\mathbb{H}^{k}(Y,\Omega_Y^{[s,t]}(f^{-1}\mathcal{L}))\rightarrow \mathbb{H}^{k-2r}(X,\Omega_X^{[s-r,t-r]}(\mathcal{L}))$ for any $k$.
Similarly, these operations can also be defined on complexes $S^{\bullet}(X,\mathcal{L},s,t)$ and $T^{\bullet}(X,\mathcal{L},s,t)$. By \cite[(3.7)]{M4}, we easily get
\begin{prop}[Projection formula]
Let $f:Y\rightarrow X$ be a proper  holomorphic map between complex manifolds and $\mathcal{L}$, $\mathcal{H}$  local systems of $\underline{\mathbb{C}}_X$-modules of finite rank on $X$. Fix integers $k$, $p$, $q$, $s$, $t$. Then
\begin{displaymath}
f_*(f^*\alpha\cup\beta)=\alpha\cup f_*\beta
\end{displaymath}
for any $\alpha\in \mathbb{H}^k(X,\Omega_X^{[s,t]}(\mathcal{L}))$ and $\beta\in H^{p,q}_{BC}(Y,f^{-1}\mathcal{H})$.
\end{prop}

Recall that a complex manifold $X$ is called \emph{$p$-K\"ahlerian}, if it admits a closed transverse positive $(p,p)$-form $\Omega$ (see \cite[Definition 1.1, 1.3]{AB}). In such case, $\Omega|_{Z}$ is a volume form on $Z$, for any complex submanifold $Z$ of pure dimension $p$ of $X$. Any complex manifold is $0$-K\"ahlerian and any K\"ahler manifold $X$ is  $p$-K\"ahlerian for every $0\leq p\leq \textrm{dim}_{\mathbb{C}}X$. We generalize  \cite[Theorem 3.1(a)(b), 4.1(a)(b)]{W}\cite[Propositions 2.3, 3.2]{M2}\cite[Proposition 3.10]{M4} as follows.

\begin{prop}\label{inj-surj}
Suppose that $f:Y\rightarrow X$ is a proper surjective holomorphic map between complex manifolds and $Y$ is $r$-K\"ahlerian, where $r=\emph{dim}_{\mathbb{C}}Y-\emph{dim}_{\mathbb{C}}X$. Let $\mathcal{L}$ be a local system of $\underline{\mathbb{C}}_X$-modules of finite rank on $X$.  Then, for any $k$, $s$, $t$,

$(1)$ $f^*:\mathbb{H}^{k}(X,\Omega_X^{[s,t]}(\mathcal{L}))\rightarrow \mathbb{H}^{k}(Y,\Omega_Y^{[s,t]}(f^{-1}\mathcal{L}))$ is injective,

$(2)$ $f_*:\mathbb{H}^{k}(Y,\Omega_Y^{[s,t]}(f^{-1}\mathcal{L}))\rightarrow \mathbb{H}^{k-2r}(X,\Omega_X^{[s-r,t-r]}(\mathcal{L}))$  is surjective.
\end{prop}
\begin{proof}
Let $\Omega$ be a strictly positive closed $(r,r)$-form on $Y$.  Then $c=f_*\Omega$ is a closed  current of degree $0$, hence a constant. By Sard's theorem, the set $X_0$ of regular values of $f$ is nonempty. For any $x\in X_0$, $Y_x=f^{-1}(x)$ is a compact complex submanifold of pure dimension $r$, so $c=\int_{Y_x}\Omega|_{Y_x}>0$. By the projection formula, $f_*(f^*\alpha\cup [\Omega]_{BC})=c\cdot \alpha$ for any $\alpha\in \mathbb{H}^{k}(X,\Omega_X^{[s,t]}(\mathcal{L}))$, where $[\Omega]_{BC}\in H_{BC}^{r,r}(X)$ denotes the Bott-Chern class of $\Omega$. It is easy to deduce the proposition.
\end{proof}

\begin{cor}
Let $f:Y\rightarrow X$ be a proper surjective holomorphic map between  complex manifolds with the same dimensions and $\mathcal{L}$ a local system of $\underline{\mathbb{C}}_X$-modules of finite rank on $X$.  Then, for any $k$, $s$, $t$

$(1)$ $f^*:\mathbb{H}^{k}(X,\Omega_X^{[s,t]}(\mathcal{L}))\rightarrow \mathbb{H}^{k}(Y,\Omega_Y^{[s,t]}(f^{-1}\mathcal{L}))$ is injective,

$(2)$ $f_*:\mathbb{H}^{k}(Y,\Omega_Y^{[s,t]}(f^{-1}\mathcal{L}))\rightarrow \mathbb{H}^{k}(X,\Omega_X^{[s,t]}(\mathcal{L}))$  is surjective.
\end{cor}

\subsection{A spectral sequence}
Associated to $S^{\bullet,\bullet}(X,\mathcal{L},s,t)$, there is a spectral sequence
\begin{equation}\label{spectral sequence}
E_1^{p,q}(X,\mathcal{L},s,t)\Rightarrow \mathbb{H}^{p+q}(X,\Omega_X^{[s,t]}(\mathcal{L})),
\end{equation}
where
\begin{equation}\label{ss1}
E_1^{p,q}(X,\mathcal{L},s,t)=H^q(S^{p,\bullet}(X,\mathcal{L},s,t))=\left\{
 \begin{array}{ll}
H^q(X,\mathcal{L}\otimes\Omega_X^p),&~s\leq p\leq t\\
 &\\
 0,&~\textrm{others}.
 \end{array}
 \right.
\end{equation}
It coincides with that associated to $T^{\bullet,\bullet}(X,\mathcal{L},s,t)$ via the natural inclusion $S^{\bullet,\bullet}(X,\mathcal{L},s,t)\rightarrow T^{\bullet,\bullet}(X,\mathcal{L},s,t)$. We call it the \emph{truncated twisted Fr\"{o}licher spectral sequence} for $(X,\mathcal{L},s,t)$. The truncated twisted Fr\"{o}licher spectral sequence for $(X,\underline{\mathbb{C}}_X,0,n)$ is just the classical Fr\"{o}licher (or Hodge-de Rham) spectral sequence.

If $X$ is \emph{compact}, one has an inequality
\begin{equation}\label{inequality}
b^k(X,\Omega_X^{[s,t]}(\mathcal{L}))\leq\sum\limits_{\substack{p+q=k\\s\leq p\leq t}}h^{p,q}(X,\mathcal{L})
\end{equation}
for any $k$, where $b^k(X,\Omega_X^{[s,t]}(\mathcal{L}))=\textrm{dim}_{\mathbb{C}}\mathbb{H}^{k}(X,\Omega_X^{[s,t]}(\mathcal{L}))$ and $h^{p,q}(X,\mathcal{L})=\textrm{dim}_{\mathbb{C}}H^q(X,\mathcal{L}\otimes\Omega_X^p)$. The spectral sequence (\ref{spectral sequence}) degenerates at $E_1$ if and only if the equalities (\ref{inequality}) hold for all $k$.

\section{Proofs of main theorems}

\subsection{Leray-Hirsch theorem}
A proof of Theorem \ref{L-H} is given as follows.
\begin{proof}
Fix two integers $s$ and $t$.
Set
\begin{displaymath}
K^{\bullet,\bullet}=\bigoplus\limits_{i=1}^{r}S^{\bullet,\bullet}(X,\mathcal{L},s-u_i,t-u_i)[-u_i,-v_i]
\end{displaymath}
and $L^{\bullet,\bullet}=S^{\bullet,\bullet}(E,\pi^{-1}\mathcal{L},s,t)$.
By (\ref{ss1}), we get the first pages
\begin{displaymath}
\begin{aligned}
_{K}E_1^{p,q}=&\bigoplus\limits_{i=1}^{r}H^{q-v_i}(S^{p-u_i,\bullet}(X,\mathcal{L},s-u_i,t-u_i))\\
=&\left\{
 \begin{array}{ll}
\bigoplus\limits_{i=1}^{r}H^{q-v_i}(X,\mathcal{L}\otimes\Omega_X^{p-u_i}),&~s\leq p\leq t\\
 &\\
 0,&~\textrm{others}
 \end{array}
 \right.
\end{aligned}
\end{displaymath}
and
\begin{displaymath}
_{L}E_1^{p,q}=\left\{
 \begin{array}{ll}
H^{q}(E,\pi^{-1}\mathcal{L}\otimes\Omega_E^p),&~s\leq p\leq t\\
 &\\
 0,&~\textrm{others}
 \end{array}
 \right.
\end{displaymath}
of the spectral sequences associated to $K^{\bullet,\bullet}$ and $L^{\bullet,\bullet}$ respectively.
By \cite[Theorem 5.6 (2)]{M4}, the morphism $\sum\limits_{i=1}^{r}\pi^*(\bullet)\wedge t_i:K^{\bullet,\bullet}\rightarrow L^{\bullet,\bullet}$
of double complexes induces an isomorphism  $_KE_1^{\bullet,\bullet}\rightarrow\mbox{ }  _LE_1^{\bullet,\bullet}$  at $E_1$-pages, hence induces an isomorphism $H^{k}(sK^{\bullet,\bullet})\rightarrow H^{k}(sL^{\bullet,\bullet})$ for any $k$.
Notice that
\begin{displaymath}
\begin{aligned}
sK^{\bullet,\bullet}=&\bigoplus_{i=1}^{r}s\left(S^{\bullet,\bullet}(X,\mathcal{L},s-u_i,t-u_i)[-u_i,-v_i]\right)\\
=& \bigoplus\limits_{i=1}^{r}S^{\bullet}(X,\mathcal{L},s-u_i,t-u_i)[-u_i-v_i]
\end{aligned}
\end{displaymath}
and $sL^{\bullet,\bullet}=S^{\bullet}(E,\pi^{-1}\mathcal{L},s,t)$.
By (\ref{computation1}),
\begin{displaymath}
\begin{aligned}
H^{k}(sK^{\bullet,\bullet})=&\bigoplus_{i=1}^{r}H^{k-u_i-v_i}(S^{\bullet}(X,\mathcal{L},s-u_i,t-u_i))\\
\cong& \bigoplus_{i=1}^{r}\mathbb{H}^{k-u_i-v_i}(X,\Omega_X^{[s-u_i,t-u_i]}(\mathcal{L}))
\end{aligned}
\end{displaymath}
and
\begin{displaymath}
H^{k}(sL^{\bullet,\bullet})=H^k(S^{\bullet}(E,\pi^{-1}\mathcal{L},s,t))\cong \mathbb{H}^{k}(E,\Omega_{E}^{[s,t]}(\pi^{-1}\mathcal{L})).
\end{displaymath}
We complete the proof.
\end{proof}

We generalize \cite[Proposition 3.3]{RYY}\cite[Corollary 3.2]{M3}\cite[Proposition 5]{St2}\cite[Proposition 2]{ASTT}\cite[Corollary 4.7]{M2}\cite[Lemma 3.3]{M5} as follows.
\begin{cor}\label{proj-bundle}
Let $\pi:\mathbb{P}(E)\rightarrow X$ be the projective bundle associated to a holomorphic vector bundle $E$ on a complex manifold $X$ and $\mathcal{L}$ a local system of $\underline{\mathbb{C}}_X$-modules of finite rank on $X$. Set $\emph{rank}_{\mathbb{C}}E=r$.

$(1)$ There exists an isomorphism
\begin{displaymath}
\bigoplus\limits_{i=0}^{r-1}\mathbb{H}^{k-2i}(X,\Omega_X^{[s-i,t-i]}(\mathcal{L}))\cong \mathbb{H}^{k}(\mathbb{P}(E),\Omega_{\mathbb{P}(E)}^{[s,t]}(\pi^{-1}\mathcal{L}))
\end{displaymath}
for any $k$, $s$, $t$.

$(2)$ Suppose that $X$ is compact and $s\leq t$. Then the truncated twisted Fr\"{o}licher spectral sequence for $(\mathbb{P}(E),\pi^{-1}\mathcal{L},s,t)$ degenerates at $E_1$-page if and only if so do those  for $(X,\mathcal{L},s-r+1,t-r+1)$, $\ldots$, $(X,\mathcal{L},s,t)$.
\end{cor}
\begin{proof}
Let $u\in A^{1,1}({\mathbb{P}(E)})$ be a first Chern form of the universal line bundle $\mathcal{O}_{\mathbb{P}(E)}(-1)$ on ${\mathbb{P}(E)}$ and
$h=[u]_{\bar{\partial}}\in H_{\bar{\partial}}^{1,1}({\mathbb{P}(E)})$ its Dolbeault class. For every $x\in X$, $1$, $h$, $\ldots$, $h^{r-1}$ restricted to the fibre $\pi^{-1}(x)=\mathbb{P}(E_x)$ freely linearly generate $H_{\bar{\partial}}^{\bullet,\bullet}(\mathbb{P}(E_x))$. By Theorem \ref{L-H}, we get $(1)$.

In general, we have
\begin{displaymath}
\begin{aligned}
b^k(\mathbb{P}(E),\Omega_{\mathbb{P}(E)}^{[s,t]}(\pi^{-1}\mathcal{L}))=&\sum_{i=0}^{r-1}b^{k-2i}(X,\Omega_X^{[s-i,t-i]}(\mathcal{L}))\quad(\mbox{by Corollary \ref{proj-bundle} (1)})\\
\leq&\sum_{i=0}^{r-1}\sum\limits_{\substack{p+q=k-2i\\s-i\leq p\leq t-i}}h^{p,q}(X,\mathcal{L})\qquad(\mbox{by (\ref{inequality})})\\
=&\sum\limits_{\substack{p+q=k\\s\leq p\leq t}}\sum_{i=0}^{r-1}h^{p-i,q-i}(X,\mathcal{L})   \\
=&\sum\limits_{\substack{p+q=k\\s\leq p\leq t}}h^{p,q}(\mathbb{P}(E),\pi^{-1}\mathcal{L})\qquad(\mbox{by \cite[Lemma 3.3]{RYY2} or \cite[Corollary 5.7(2) ]{M4}})
\end{aligned}
\end{displaymath}
So,
\begin{displaymath}
b^k(\mathbb{P}(E),\Omega_{\mathbb{P}(E)}^{[s,t]}(\pi^{-1}\mathcal{L}))=\sum\limits_{\substack{p+q=k\\s\leq p\leq t}}h^{p,q}(\mathbb{P}(E),\pi^{-1}\mathcal{L})\mbox{ for all $k$,}
\end{displaymath}
if and only if,
\begin{displaymath}
b^{k}(X,\Omega_X^{[s-i,t-i]}(\mathcal{L}))=\sum\limits_{\substack{p+q=k\\s-i\leq p\leq t-i}}h^{p,q}(X,\mathcal{L})\mbox{ for all $k$ and $0\leq i\leq r-1$.  }
\end{displaymath}
Thus $(2)$ follows.
\end{proof}

\begin{rem}
From the proof of \cite[Corollary 5.7]{M3}, we may similarly obtain the flag bundle formula of the  hypercohomology of  truncated twisted holomorphic de Rham complexes. Of course, its expression is much more sophisticated.
\end{rem}

For an $n$-dimensional complex manifold $X$, the inclusion $\Omega_X^{[p,n]}(\underline{\mathbb{C}}_X)\rightarrow  \Omega_X^{[0,n]}(\underline{\mathbb{C}}_X)$ induces the Hodge filtration on $H_{dR}^k(X,\mathbb{C})$ (\cite[Definition 8.2]{V}) as
\begin{displaymath}
F^pH_{dR}^k(X,\mathbb{C})=Im(\mathbb{H}^k(X,\Omega_X^{[p,n]}(\underline{\mathbb{C}}_X))\rightarrow \mathbb{H}^k(X,\Omega_X^{[0,n]}(\underline{\mathbb{C}}_X))=H_{dR}^k(X,\mathbb{C}))
\end{displaymath}
for any $k$ and $p$. By (\ref{computation1}) and (\ref{computation2}), the inclusions $S^{\bullet}(X,\underline{\mathbb{C}}_X,p,n)\rightarrow S^{\bullet}(X,\underline{\mathbb{C}}_X,0,n)$ and $T^{\bullet}(X,\underline{\mathbb{C}}_X,p,n)\rightarrow T^{\bullet}(X,\underline{\mathbb{C}}_X,0,n)$ induce
\begin{equation}\label{filtration1}
F^pH_{dR}^k(X,\mathbb{C})=Im(H^k(S^{\bullet}(X,\underline{\mathbb{C}}_X,p,n))\rightarrow H^k(S^{\bullet}(X,\underline{\mathbb{C}}_X,0,n)))
\end{equation}
and
\begin{equation}\label{filtration2}
F^pH_{dR}^k(X,\mathbb{C})=Im(H^k(T^{\bullet}(X,\underline{\mathbb{C}}_X,p,n))\rightarrow H^k(T^{\bullet}(X,\underline{\mathbb{C}}_X,0,n))),
\end{equation}
respectively.
\begin{cor}[{\cite[Lemma 3.4]{M5}}]\label{Hodge-filtration}
Let $\pi:\mathbb{P}(E)\rightarrow X$ be the projective bundle associated to a holomorphic vector bundle $E$ on a complex manifold $X$ and $\emph{rank}_{\mathbb{C}}E=r$.
Suppose that $u\in \mathcal{A}^{1,1}(\mathbb{P}(E))$ is a first Chern form of the universal line bundle $\mathcal{O}_{\mathbb{P}(E)}(-1)$ on $\mathbb{P}(E)$. Then
\begin{displaymath}
\sum\limits_{i=0}^{r-1}\pi^*(\bullet)\wedge u^i
\end{displaymath}
gives an isomorphism
\begin{displaymath}
\bigoplus_{i=0}^{r-1}F^{p-i}H_{dR}^{k-2i}(X,\mathbb{C})\tilde{\rightarrow} F^pH_{dR}^k(\mathbb{P}(E),\mathbb{C}),
\end{displaymath}
for any $k$, $p$.
\end{cor}
\begin{proof}
Set $\textrm{dim}_{\mathbb{C}}X=n$. Let $u\in \mathcal{A}^{1,1}({\mathbb{P}(E)})$ be a first Chern form of the universal line bundle $\mathcal{O}_{\mathbb{P}(E)}(-1)$ on ${\mathbb{P}(E)}$. Then $\partial u=\bar{\partial}u=0$.
Consider the commutative diagram of complexes
\begin{displaymath}
\xymatrix{
\bigoplus\limits_{i=0}^{r-1}S^{\bullet}(X,\underline{\mathbb{C}}_X,p-i,n+r-1-i)[-2i] \qquad\ar[d]_{} \ar[r]^{\qquad\sum\limits_{i=0}^{r-1}\pi^*(\bullet)\wedge u^i}  &\quad S^{\bullet}(\mathbb{P}(E),\underline{\mathbb{C}}_{\mathbb{P}(E)},p,n+r-1) \ar[d]^{} \\
\bigoplus\limits_{i=0}^{r-1}S^{\bullet}(X,\underline{\mathbb{C}}_X,-i,n+r-1-i)[-2i]\qquad \ar[r]^{\qquad\sum\limits_{i=0}^{r-1}\pi^*(\bullet)\wedge u^i}  &\quad S^{\bullet}(\mathbb{P}(E),\underline{\mathbb{C}}_{\mathbb{P}(E)},0,n+r-1).}
\end{displaymath}
Notice that $S^{\bullet}(X,\underline{\mathbb{C}}_X,p-i,n+r-1-i)=S^{\bullet}(X,\underline{\mathbb{C}}_X,p-i,n)$ and $S^{\bullet}(X,\underline{\mathbb{C}}_X,-i,n+r-1-i)=S^{\bullet}(X,\underline{\mathbb{C}}_X,0,n)$ for $0\leq i\leq r-1$. Thus we have  the commutative diagram
\begin{displaymath}
\xymatrix{
 \bigoplus\limits_{i=0}^{r-1}H^{k-2i}(S^{\bullet}(X,\underline{\mathbb{C}}_X,p-i,n))\ar[d]_{} \ar[r]^{\cong\quad}  & H^k(S^{\bullet}(\mathbb{P}(E),\underline{\mathbb{C}}_{\mathbb{P}(E)},p,n+r-1)) \ar[d]^{} \\
\bigoplus\limits_{i=0}^{r-1}H^{k-2i}(S^{\bullet}(X,\underline{\mathbb{C}}_X,0,n)) \ar[r]^{\cong\qquad}  & H^k(S^{\bullet}(\mathbb{P}(E),\underline{\mathbb{C}}_{\mathbb{P}(E)},0,n+r-1)),}
\end{displaymath}
where the two horizontal maps are isomorphisms by Corollary \ref{proj-bundle} $(1)$. Consequently, we get the corollary by (\ref{filtration1}).
\end{proof}

\subsection{K\"{u}nneth theorem}
Let $X$, $Y$ be complex manifolds and  let $pr_1$, $pr_2$ be projections from $X\times Y$ onto $X$, $Y$, respectively. For sheaves $\mathcal{L}$ and $\mathcal{H}$ of $\underline{\mathbb{C}}_X$- and $\underline{\mathbb{C}}_Y$-modules  on $X$ and $Y$ respectively,  the \emph{external tensor product} of $\mathcal{L}$ and $\mathcal{H}$ on $X\times Y$ is defined as
\begin{displaymath}
\mathcal{L}\boxtimes\mathcal{H}=pr_1^{-1}\mathcal{L}\otimes_{\underline{\mathbb{C}}_{X\times Y}}pr_2^{-1}\mathcal{H}.
\end{displaymath}
For coherent analytic sheaves $\mathcal{F}$ and $\mathcal{G}$ of $\mathcal{O}_X$- and $\mathcal{O}_Y$-modules  on $X$ and $Y$ respectively,  the \emph{analytic external tensor product} of $\mathcal{F}$ and $\mathcal{G}$ on $X\times Y$  is defined as
\begin{displaymath}
\mathcal{F}\boxtimes\mathcal{G}=pr_1^*\mathcal{F}\otimes_{\mathcal{O}_{X\times Y}}pr_2^*\mathcal{G}.
\end{displaymath}

Now, we verify Theorem \ref{K}.
\begin{proof}
Fix two integers $s$ and $t$.
Consider the double complexes
\begin{displaymath}
K^{\bullet,\bullet}=\bigoplus\limits_{\substack{u+w=s\\v+w=t}}ss(S^{\bullet,\bullet}(X,\mathcal{L},u,v)\otimes_{\mathbb{C}} S^{\bullet,\bullet}(Y,\mathcal{H},w,w)),
\end{displaymath}
\begin{displaymath}
L^{\bullet,\bullet}=S^{\bullet,\bullet}(X\times Y,\mathcal{L}\boxtimes \mathcal{H},s,t)
\end{displaymath}
and a morphism $f=pr_1^*(\bullet)\wedge pr_2^*(\bullet):K^{\bullet,\bullet}\rightarrow L^{\bullet,\bullet}$.
The first pages of the spectral sequences associated to $K^{\bullet,\bullet}$ and $L^{\bullet,\bullet}$ are calculated as follows
\begin{displaymath}
\begin{aligned}
_{K}E_1^{a,b}
=&\bigoplus\limits_{\substack{u+w=s\\v+w=t}}\bigoplus\limits_{\substack{p+q=a\\k+l=b}}H^{k}(S^{p,\bullet}(X,\mathcal{L},u,v))\otimes_{\mathbb{C}} H^{l}(S^{q,\bullet}(Y,\mathcal{H},w,w))\quad\mbox{ }(\mbox{by Proposition \ref{cohomology of complexes} (2)})\\
=&\bigoplus\limits_{\substack{p+q=a\\k+l=b}}\bigoplus\limits_{\substack{u+w=s\\v+w=t\\u\leq p\leq v\\w=q}}H^{k}(S^{p,\bullet}(X,\mathcal{L},u,v))\otimes_{\mathbb{C}} H^{l}(S^{q,\bullet}(Y,\mathcal{H},w,w))\quad\mbox{ }(\mbox{by (\ref{ss1})})\\
=&\bigoplus\limits_{\substack{p+q=a\\k+l=b}}\bigoplus\limits_{\substack{s-q\leq p\leq t-q}}H^{k}(X,\mathcal{L}\otimes\Omega_X^p)\otimes_{\mathbb{C}} H^{l}(Y,\mathcal{H}\otimes\Omega_Y^q)\qquad\qquad(\mbox{by (\ref{ss1})})\\
=&\left\{
 \begin{array}{ll}
\bigoplus\limits_{\substack{p+q=a\\k+l=b}}H^{k}(X,\mathcal{L}\otimes\Omega_X^p)\otimes_{\mathbb{C}} H^{l}(Y,\mathcal{H}\otimes\Omega_Y^q),&~s\leq a\leq t\\
 &\\
 0,&~\textrm{others}
 \end{array}
 \right.
\end{aligned}
\end{displaymath}
and
\begin{displaymath}
_{L}E_1^{a,b}=\left\{
 \begin{array}{ll}
H^{b}(X\times Y,(\mathcal{L}\boxtimes \mathcal{H})\otimes\Omega_{X\times Y}^a),&~s\leq a\leq t\\
 &\qquad\qquad\qquad(\mbox{by Proposition \ref{cohomology of complexes} (2) and (\ref{ss1})})\\
 0,&~\textrm{others}
 \end{array}.
 \right.
\end{displaymath}
The morphism $E^{a,b}_1(f): _{K}E_1^{a,b}\rightarrow _{L}E_1^{a,b}$ at $E_1$-pages induced by $f$ is just the cartesian product
\begin{displaymath}
\bigoplus\limits_{\substack{p+q=a\\k+l=b}}H^{k}(X,\mathcal{L}\otimes\Omega_X^p)\otimes_{\mathbb{C}} H^{l}(Y,\mathcal{H}\otimes\Omega_Y^q)\rightarrow H^{b}(X\times Y,(\mathcal{L}\boxtimes \mathcal{H})\otimes\Omega_{X\times Y}^a)
\end{displaymath}
for $s\leq a\leq t$ and the identity between zero spaces for other cases.
Notice that $(\mathcal{L}\otimes\Omega_X^p)\boxtimes (\mathcal{H}\otimes\Omega_Y^q)=(\mathcal{L}\boxtimes \mathcal{H})\otimes (\Omega_X^p\boxtimes \Omega_Y^q)$ and $\Omega_{X\times Y}^a=\bigoplus\limits_{p+q=a}(\Omega_{X}^p\boxtimes\Omega_{Y}^q)$.
By \cite[IX, (5.23) (5.24)]{Dem}, $E^{a,b}_1(f)$ is an isomorphism for any $a$, $b$, so is the morphism $H^{k}(sK^{\bullet,\bullet})\rightarrow H^{k}(sL^{\bullet,\bullet})$ induced by $f$ for any $k$.
By Proposition \ref{cohomology of complexes} (1) and (\ref{computation1}),
\begin{displaymath}
H^{c}(sK^{\bullet,\bullet})\cong \bigoplus\limits_{\substack{a+b=c\\u+w=s\\v+w=t}}\mathbb{H}^{a}(X,\Omega_X^{[u,v]}(\mathcal{L}))\otimes \mathbb{H}^{b}(Y,\Omega_Y^{[w,w]}(\mathcal{H}))
\end{displaymath}
and $H^{c}(sL^{\bullet,\bullet})\cong\mathbb{H}^{c}(X\times Y,\Omega_{X\times Y}^{[s,t]}(\mathcal{L}\boxtimes\mathcal{H}))$.
We conclude this  theorem.
\end{proof}

\subsection{Poincar\'{e}-Serre duality theorem}
Let $X$ be a  complex manifold and $\mathcal{L}$ a local system of $\underline{\mathbb{C}}_X$-modules of rank $l$ on $X$.
Denote by $\mathcal{L}^{\vee}=\mathcal{H}om_{\underline{\mathbb{C}}_X}(\mathcal{L},\underline{\mathbb{C}}_X)$ the dual of $\mathcal{L}$.
For $\gamma\in\Gamma(X,\mathcal{L}\otimes\mathcal{L}^\vee\otimes\mathcal{A}_X^{p})$, we construct  $tr(\gamma)$ as follows:
Suppose that $U$ is an open subset of $X$ such that $\mathcal{L}|_U$ is trivial.
Let $e_1$, $\ldots$, $e_l$ and $f_1$, $\ldots$, $f_l$ be bases of $\Gamma(U,\mathcal{L})$ and $\Gamma(U, \mathcal{L}^{\vee})$, respectively.
Set $\gamma=\sum\limits_{\substack{1\leq i\leq l\\1\leq j\leq l}}e_i\otimes f_i\otimes\gamma_{ij}$  on $U$, where $\gamma_{ij}\in\mathcal{A}^{p}(U)$ for any $1\leq i, j\leq l$.
Then $tr(\gamma)$ on $U$ is defined as
$\sum\limits_{\substack{1\leq i\leq l\\1\leq j\leq l}} \langle e_i, f_j\rangle\gamma_{ij}$, where $\langle,\rangle$ is the contraction between $\mathcal{L}$ and $\mathcal{L}^{\vee}$.
This construction is global.
By the construction, the action $tr$ on $\gamma\in\Gamma(X,\mathcal{L}\otimes\mathcal{L}^\vee\otimes\mathcal{A}_X^{p})$  is just contracting the parts of $\mathcal{L}$ and $\mathcal{L}^{\vee}$ and preserving the part of the form of $\gamma$, which was essentially defined in \cite{Se} (even in more general cases).

Assume that $\mathcal{H}$ is a local system of $\underline{\mathbb{C}}_X$-modules of finite rank on $X$. For $\alpha\in\Gamma(X,\mathcal{L}\otimes\mathcal{A}_X^{p})$ and $\beta\in\Gamma(X,\mathcal{H}\otimes\mathcal{A}_X^{q})$,
\begin{equation}\label{differential of product}
d(tr(\alpha\wedge\beta))=tr(d\alpha\wedge\beta)+(-1)^ptr(\alpha\wedge d\beta),
\end{equation}
which also holds if using $\partial$ or $\bar{\partial}$ instead of $d$.

Theorem \ref{P-S} is shown as follows.
\begin{proof}
Fix two integers $s$ and $t$.
Set 
\begin{displaymath}
K^{\bullet,\bullet}=S^{\bullet,\bullet}(X,\mathcal{L},s,t) \mbox{ and } L^{\bullet,\bullet}=\left(S^{\bullet,\bullet}(X,\mathcal{L}^{\vee},n-t,n-s)\right)^*[-n,-n].
\end{displaymath}
By (\ref{ss1}) and Lemma \ref{Duality}, we have
\begin{displaymath}
_{K}E_1^{p,q}=\left\{
 \begin{array}{ll}
H^{q}(X,\mathcal{L}\otimes\Omega_X^p),&~s\leq p\leq t\\
 &\\
 0,&~\textrm{others}
 \end{array}
 \right.
\end{displaymath}
and 
\begin{displaymath}
\begin{aligned}
_{L}E_1^{p,q}=&H^{q-n}((S^{\bullet,\bullet}(X,\mathcal{L}^\vee,n-t,n-s))^{*p-n,\bullet})\\
=&H^{q-n}((S^{n-p,\bullet}(X,\mathcal{L}^\vee,n-t,n-s))^*)\\
=&\left\{
 \begin{array}{ll}
(H^{n-q}(X,\mathcal{L}^\vee\otimes \Omega_X^{n-p}))^*,&~s\leq p\leq t\\
 &\\
 0,&~\textrm{others}.
 \end{array}
 \right.
\end{aligned}
\end{displaymath}
The map
\begin{displaymath}
\mathcal{D}:K^{\bullet,\bullet}\rightarrow L^{\bullet,\bullet}, \mbox{ } \alpha\mapsto \int_X tr(\alpha\wedge\bullet)
\end{displaymath}
is a  morphism of double complexes by (\ref{differential of product}).
The morphism  $E_1^{p,q}(\mathcal{D}): _KE_1^{p,q}\rightarrow _LE_1^{p,q}$  at $E_1$-pages induced by $\mathcal{D}$ is just the Serre duality map  $H^{q}(X,\mathcal{L}\otimes\Omega_X^p)\rightarrow (H^{n-q}(X,\mathcal{L}^\vee\otimes \Omega_X^{n-p}))^*$ for $s\leq p\leq t$ and the identity between zero spaces for other cases.
By Serre duality theorem, $E_1^{p,q}(\mathcal{D})$ is an isomorphism for any $p$, $q$, so is  $H^{k}(sK^{\bullet,\bullet})\rightarrow H^{k}(sL^{\bullet,\bullet})$ induced by $\mathcal{D}$ for any $k$.
By Lemma \ref{simple-duality},
\begin{displaymath}
\begin{aligned}
sL^{\bullet,\bullet}=&\left(s\left(S^{\bullet,\bullet}(X,\mathcal{L}^{\vee},n-t,n-s)^*\right)\right)[-2n]\\
\cong& S^{\bullet}(X,\mathcal{L}^{\vee},n-t,n-s)^*[-2n].
\end{aligned}
\end{displaymath}
By (\ref{computation1}) and Lemma \ref{Duality}, we have
\begin{displaymath}
H^{k}(sK^{\bullet,\bullet})\cong \mathbb{H}^{k}(X,\Omega_X^{[s,t]}(\mathcal{L}))
\mbox{ and }
H^{k}(sL^{\bullet,\bullet})\cong \left(\mathbb{H}^{2n-k}(X,\Omega_X^{[n-t,n-s]}(\mathcal{L}^{\vee}))\right)^*,
\end{displaymath}
from which Theorem \ref{P-S} follows.
\end{proof}

\begin{rem}
Theorem \ref{P-S} is the Poincare duality theorem if  $s=0$, $t=n$ and a special case of the Serre duality theorem if $s=t$.
\end{rem}

\subsection{Blowup formula}
Let $\pi:\widetilde{X}\rightarrow X$ be the blowup of a  complex manifold $X$ along a  complex submaifold $Y$ and $\mathcal{L}$ a local system of $\underline{\mathbb{C}}_X$-modules of finite rank on $X$.
As we know, $\pi|_E:E=\pi^{-1}(Y)\rightarrow Y$ can be naturally viewed as the projective bundle $\mathbb{P}(N_{Y/X})$ associated to the normal bundle $N_{Y/X}$ of $Y$ in $X$.
Let $u\in \mathcal{A}^{1,1}(E)$ be a first Chern form of the universal line bundle $\mathcal{O}_{E}(-1)$ on $E\cong{\mathbb{P}(N_{Y/X})}$.
Suppose that $i_Y:Y\rightarrow X$ and  $i_E:E\rightarrow \widetilde{X}$ are the inclusions.  Set $r=\textrm{codim}_{\mathbb{C}}Y\geq 2$.

By \cite[Theorem 1.2]{M4}, Theorem  \ref{blowup} can be similarly proved  with  Theorems \ref{L-H}. Now, we provide an alternative proof as follows.
\begin{proof}
Consider the complexes
\begin{displaymath}
K^{\bullet}(s,t)=S^{\bullet}(X,\mathcal{L},s,t)\oplus\bigoplus_{i=1}^{r-1}S^{\bullet}(Y,i_Y^{-1}\mathcal{L},s-i,t-i)[-2i]
\end{displaymath}
and $L^{\bullet}(s,t)=T^{\bullet}(\widetilde{X},\pi^{-1}\mathcal{L},s,t)$ and the morphism
\begin{displaymath}
f(s,t)=\pi^*+\sum\limits_{i=1}^{r-1}i_{E*}\circ (u^{i-1}\wedge)\circ (\pi|_{E})^*:K^{\bullet}(s,t)\rightarrow L^{\bullet}(s,t).
\end{displaymath}
By (\ref{computation1}) and  (\ref{computation2}), the left side and the right side of (\ref{blowup0}) are $H^k(K^\bullet(s,t))$ and  $H^k(L^\bullet(s,t))$ respectively. Our goal is to show that $H^k(f(s,t)):H^k(K^{\bullet}(s,t))\rightarrow H^k(L^{\bullet}(s,t))$ is an isomorphism for any $s$, $t$.

We may assume that $0\leq s\leq t$. By  (\ref{exact}),  there is the commutative diagram
\begin{displaymath}
\xymatrix{
0\ar[r]&K^{\bullet}(s,t)\ar[d]^{f(s,t)}\ar[r]^{} & K^{\bullet}(0,t)\ar[d]^{f(0,t)}\ar[r]^{}& K^{\bullet}(0,s-1) \ar[d]^{f(0,s-1)}\ar[r]& 0\\
0\ar[r]&L^{\bullet}(s,t)    \ar[r]^{}& L^{\bullet}(0,t)\ar[r]^{} & L^{\bullet}(0,s-1)     \ar[r]& 0}
\end{displaymath}
of short exact sequences of complexes, which induces the commutative diagram
\begin{equation}\label{long exact}
\tiny{\xymatrix{
    H^{k-1}(K^{\bullet}(0,t)) \ar[d]^{H^{k-1}(f(s,t))} \ar[r]&H^{k-1}(K^{\bullet}(0,s-1))\ar[d]^{H^{k-1}(f(0,s-1))}\ar[r]&H^k(K^{\bullet}(s,t)) \ar[d]^{H^k(f(s,t))} \ar[r]& H^k(K^{\bullet}(0,t)) \ar[d]^{H^k(f(0,t))}\ar[r]&  H^k(K^{\bullet}(0,s-1))\ar[d]^{H^k(f(0,s-1))}\\
 H^{k-1}(L^{\bullet}(0,t))       \ar[r]&H^{k-1}(L^{\bullet}(0,s-1))     \ar[r] & H^k(L^{\bullet}(s,t))\ar[r]&H^k(L^{\bullet}(0,t))       \ar[r]& H^k(L^{\bullet}(0,s-1)) }}
\end{equation}
of long exact sequences. By \cite[Theorem 1.2]{M4}, $H^\bullet(f(t,t))$ is an isomorphism for any $t$. Fisrt, consider  (\ref{long exact}) for $s=t$. By the induction on $t\geq 0$ and the five-lemma,  $H^\bullet(f(0,t))$ is an isomorphism for any $t\geq 0$. Furthermore, $H^\bullet(f(s,t))$ is an isomorphism  for any $s$, $t$, by the five-lemma  again in (\ref{long exact}) for general cases.

\end{proof}

\begin{rem}
Theorem \ref{blowup} for $s=0$, $t=n$ is just \cite[Theorem 1]{CY} and the first formula of \cite[Theorem 1.2]{M4}, while the two proofs here are quite different from those there.
\end{rem}

Following the  proof of Corollary \ref{proj-bundle} (2), we generalize \cite[Theorem 1.6]{RYY}\cite[Corollary 2]{CY} as follows.
\begin{cor}\label{c2}
Under the hypotheses of Theorem \ref{blowup}, suppose that $X$ is compact and $s\leq t$. Then the truncated twisted Fr\"{o}licher spectral sequence for $(\widetilde{X},\pi^{-1}\mathcal{L},s,t)$ degenerates at $E_1$-page if and only if so do those for  $(X,\mathcal{L},s,t)$, $(Y,i_Y^{-1}\mathcal{L},s-r+1,t-r+1)$, $\ldots$, $(Y,i_Y^{-1}\mathcal{L},s-1,t-1)$.
\end{cor}

With the similar proof of Corollary \ref{Hodge-filtration}, we easily get the Hodge filtration of a blowup as follows.
\begin{cor}[{\cite[Lemma 3.7]{M5}}]
Under the hypotheses of Theorem \ref{blowup},
\begin{displaymath}
\pi^*+\sum\limits_{i=1}^{r-1}i_{E*}\circ (u^{i-1}\wedge)\circ (\pi|_{E})^*
\end{displaymath}
gives an isomorphism
\begin{displaymath}
F^pH_{dR}^k(X,\mathbb{C})\oplus\bigoplus_{i=1}^{r-1}F^{p-i}H_{dR}^{k-2i}(Y,\mathbb{C})\tilde{\rightarrow} F^pH_{dR}^k(\widetilde{X},\mathbb{C}),
\end{displaymath}
for any $k$, $p$.
\end{cor}

\subsection{A remark}
In the proofs of Theorems \ref{L-H}-\ref{P-S} and the first proof of Theorem \ref{blowup}, we use almost the same steps, that is, first proving that the morphisms between double complexes induce isomorphisms at $E_1$-pages and then obtaining their induced isomorphisms at $H$-pages, where we need the corresponding results (respectively, the Leray-Hirsch, K\"{u}nneth, Serre duality theorems and the blowup formula) on the twisted Dolbeault cohomology in the former steps. In the second proof of Theorem \ref{blowup}, we use the blowup formula on the twisted Dolbeault cohomology and the five-lemma. So  the twisted Dolbeault cohomology and some algebraic machines  play significant  roles in the research of the hypercohomology of truncated twisted holomorphic de Rham complexes. For the twisted Dolbeault cohomology, we refer to \cite{M4,RYY2,W}.

%



\begin{thebibliography}{GP}

\bibitem{AB}
Alessandrini, L., Bassanelli, G.: Compact $p$-K\"ahler manifolds. Geom. Dedicata \textbf{38}, 199-210 (1991)


\bibitem{ASTT}
Angella, D., Suwa, T., Tardini, N., Tomassini, A.: Note on Dolbeault cohomology and Hodge structures up to bimeromorphisms. ArXiv: 1712.08889v2(25 March, 2019)


\bibitem{CY}
Chen, Y., Yang, S.: On the blow-up formula of twisted de Rham cohomology. Ann. Global Anal. Geom. \textbf{56}, 277-290 (2019)


\bibitem{Dem}
Demailly, J.-P.: \textit{Complex Analytic and Differential Geometry},
\href{http://www-fourier.ujf-grenoble.fr/~demailly/documents.html}{http://www-fourier.ujf-grenoble.fr/~demailly/documents.html}







\bibitem{M1}
Meng, L.: Morse-Novikov cohomology for blow-ups of complex manifolds. ArXiv:1806.06622v4 (19 October, 2019)

\bibitem{M2}
Meng, L.: Morse-Novikov cohomology on complex manifolds. J. Geom. Anal. \textbf{30}, 493-510 (2020)

\bibitem{M3}
Meng, L.: Leray-Hirsch theorem and blow-up formula for Dolbeault cohomology. To appear in Ann. Mat. Pura Appl. (4),
\href{https://doi.org/10.1007/s10231-020-00953-y}{https://doi.org/10.1007/s10231-020-00953-y}.
ArXiv:1806.11435v6 (2 February, 2020)

\bibitem{M4}
Meng, L.: Mayer-Vietoris systems and their applications. ArXiv:1811.10500v3 (17 April, 2019)

\bibitem{M5}
Meng, L.: Three theorem on the $\partial\bar{\partial}$-lemma. ArXiv:1905.13585v1 (31 May, 2019)



\bibitem{RYY}
Rao, S., Yang, S., Yang, X.-D.:
Dolbeault cohomologies of blowing up complex manifolds. J. Math. Pures Appl. (9) \textbf{130}, 68-92 (2019)

\bibitem{RYY2}
Rao, S., Yang, S., Yang, X.-D.:
Dolbeault cohomologies of blowing up complex manifolds II: bundle-valued case. J. Math. Pures Appl. (9) \textbf{133}, 1-38 (2020)

\bibitem{Se}
Serre, J.-P.: Un th\'{e}or\`{e}me de dualit\'{e}. Comment. Math. Helv. \textbf{29}, 9-26 (1955)


\bibitem{St2}
Stelzig, J.:  The double complex of a blow-up. To appear in Int. Math. Res. Not. IMRN.,
\href{https://doi.org/10.1093/imrn/rnz139}{https://doi.org/10.1093/imrn/rnz139}.
ArXiv:1808.02882v2 (3 June, 2019)










\bibitem{V}
Voisin, C.: \textit{Hodge Theory and Complex Algebraic Geometry. Vol. I}. Cambridge Stud. Adv. Math. \textbf{76}, Cambridge University Press, Cambridge, 2003

\bibitem{W}
Wells, R. O.: Comparison of de Rham and Dolbeault cohomology for proper surjective mappings. Pacific J. Math. \textbf{53}, 281-300 (1974)


\bibitem{YZ}
Yang, X.-D., Zhao, G.: A note on the Morse-Novikov cohomology of blow-ups of locally conformal K\"{a}hler manifolds. Bull. Aust. Math. Soc. \textbf{91}, 155-166 (2015)

\bibitem{Z}
Zou, Y.: On the Morse-Novikov cohomology of blowing up complex manifolds. ArXiv:1907.13336v1 (31 July, 2019)


\end{thebibliography}
\end{document}